\documentclass[12pt,reqno]{amsart}
\usepackage{amssymb}
\usepackage{amsmath, mathtools}

\usepackage{amscd}

\newcommand{\RNum}[1]{\uppercase\expandafter{\romannumeral #1\relax}}

\usepackage[T2A]{fontenc}
\usepackage[utf8]{inputenc}
\usepackage[english]{babel}

\input{int.def} 
\DeclarePairedDelimiter \abs{\lvert}{\rvert} 

\DeclarePairedDelimiterX \ip[2]{\langle}{\rangle}{#1,#2}
\DeclarePairedDelimiterXPP \Prob[1]{\mathbb{P}}\{\}{}{
   
   #1} 
\DeclarePairedDelimiterXPP \Probevent[1]{\mathbb{P}}(){}{
   
   #1} 

\DeclareMathOperator{\E}{\mathbb{E}}

\def \R {\mathbb{R}}

\usepackage{dirtytalk}
\usepackage{hyperref}
\usepackage{xcolor}
\usepackage{centernot}

\usepackage{enumitem}

\usepackage{pgfplots}
\usepackage{multicol}

\pgfplotsset{compat=1.17}

\numberwithin{equation}{section}

\usepackage[mathcal]{euscript}

\usepackage{titlesec}
\titleformat{\section}[runin]{\bfseries}{\thesection.}{3pt}{}[.]

\usepackage{geometry}
\newgeometry{vmargin={25mm}, hmargin={22mm,22mm}, footskip=10mm}   

\begin{document}

\title[Improving discrepancy by moving a few points]{Improving discrepancy by moving a few points}

\author{Gleb Smirnov}
\address{
Mathematical Sciences Institute, 
Australian National University, Canberra, Australia}
\email{gleb.smirnov@anu.edu.au}

\author{Roman Vershynin}
\address{Department of Mathematics, University of California, Irvine, U.S.A.}
\curraddr{}
\email{rvershyn@uci.edu}

\thanks{R.V.~is partially supported by NSF Grant DMS 1954233 and NSF+Simons Research Collaborations on the Mathematical and Scientific Foundations of Deep Learning}


\begin{abstract}

We show how to improve the discrepancy of an iid sample by moving only a few points. Specifically, modifying \( O(m) \) sample points on average reduces the Kolmogorov distance to the population distribution 
to \(O(1/m)\).
\end{abstract}

\maketitle

\setcounter{section}{0}

\section{Main results}\label{intro}

Let \( X_1, \dots, X_n \) be an iid sample from a distribution on $\R$ with cumulative distribution function \( F(x) \). The empirical distribution function of the sample is defined as
\[
F_n(x) = \frac{1}{n} \sum_{i=1}^{n} \mathbf{1}_{\{X_i \leq x\}},
\quad x \in \R.
\]
The sample discrepancy \( D_n \) is defined as the Kolmogorov distance between the empirical and population distributions: 
\[
D_n = D_n(X_1,\ldots, X_n) = \sup_x \abs{F_n(x) - F(x)}.
\]
Dvoretzky–Kiefer–Wolfowitz inequality (DKW) implies that $D_n = O(1/\sqrt{n})$ with high probability \cite{D-K-W, Massart}. Our main result shows that the discrepancy can be significantly reduced by carefully changing a small fraction of the sample:

\begin{theorem}[Improving discrepancy]\label{1sample}
    Given iid  observations \( X_1, \dots, X_n \) from a distribution with cumulative distribution function \( F(x) \) and a number  $m \in (0,n]$, there is a randomized algorithm to replace at most \( C m \) observations on average so that
    \begin{equation}\label{Dn}
        D_n = \sup_x \abs{F_n(x) - F(x)} \leq \frac{2}{m},
    \end{equation}
    where $C > 0$ is an absolute constant.
\end{theorem}

More precisely, Theorem~\ref{1sample} states that there is a randomized algorithm (Algorithm~1 below) that produces modified data \(\tilde{X}_1, \dots, \tilde{X}_n\) such that
\[
\E \left[\sum_{i=1}^n \mathbf{1}_{\{X_i \neq \tilde{X}_i\}}\right] \le Cm 
\quad \text{and} \quad 
D_n(\tilde{X}_1, \dots, \tilde{X}_n) \leq \frac{2}{m} 
\quad \text{deterministically},
\]
with the expectation taken over both the sample and the algorithm. The algorithm has time complexity of \(O\bigl(n\ln(n)\ln(m)\bigr)\), see \S\,\ref{complexity}.
\smallskip

Theorem~\ref{1sample} is sharp in general, as the following matching lower bound shows:

\begin{proposition}[Lower bound]\label{lower_bound}
    Let \( X_1, \dots, X_n \) be iid observations from the uniform distribution on $[0,1]$ and let $m>0$ be a number such that $m^2/n>1$ is an integer. Then achieving a discrepancy of \( c/m \) requires replacing at least $C m$ observations on average. Here $c$ is an absolute constant. 
\end{proposition}

More precisely, Proposition~\ref{lower_bound} says that 
\[
\mathbb{E}\,\inf_{(\tilde{X}_1,\dots,\tilde{X}_n)\, 
\in\, S_{c/m}}
\Biggl[\sum_{i=1}^n \mathbf{1}_{\{X_i \neq \tilde{X}_i\}}\Biggr] \;\geq\; C m,
\]
where $S_{c/m}$ is the set of sequences with discrepancy at most \(c/m\):  
\[
S_{c/m} =
\Bigl\{
\tilde{X}_1,\dots,\tilde{X}_n :
D_n(\tilde{X}_1,\dots,\tilde{X}_n) \le c/m
\Bigr\}.
\]
The expectation is taken over the sample \(X_1, \dots, X_n\).
\smallskip%

The constraint $m^2/n>1$ in Proposition~\ref{lower_bound} is explained by the fact that the discrepancy is always $O(1/\sqrt{n})$ with high probability by the DKW inequality, even without replacing any points.

\begin{remark}[Knowledge of the distribution]
The algorithm in Theorem~\ref{1sample} needs to know the underlying distribution. This is unavoidable if $m \gg \sqrt{n}$ due to the basic statistical fact that the cumulative distribution function cannot be estimated from an iid sample of size $n$ with error smaller than $O(1/\sqrt{n})$.
\end{remark}

\begin{remark}[Removing instead of replacing]
    This work is inspired by the recent result of D.~Bilyk and S.~Steinerberger \cite{B-B}, which achieves a discrepancy of \( O(\ln(n)/m) \) by \emph{removing} at most \( m \) points from the original sample. As shown in \cite{B-B}, this result is optimal in terms of the number of removed points. Surprisingly, Theorem~\ref{1sample} shows how that the logarithmic overhead is unnecessary if we allow to replace rather than remove the points. This shows that the sample can be regularized more efficiently than previously known (see, e.g., \cite{DFGGR, B-B}). In discrepancy theory, logarithmic factors often emerge naturally from chaining arguments and union bounds, so avoiding them is unexpected. 
\end{remark}

\begin{remark}[Other work]
If we are allowed to choose {\em all} points $X_1,\ldots,X_n$, then our problem becomes straightforward: placing the points at the quantiles of the distribution achieves discrepancy $1/n$. This is much less trivial in higher dimensions (which this work does not touch), and there is a wealth of work aiming at finding configurations of points that minimize discrepancies, such as 
star discrepancy and Wasserstein distance 
\cite{Mat, Chaz, Beck-Chen, B-B, 
DFGGR, Aist-Dick, Aist-B-Nik, March, Quatt}.
\smallskip%

High-dimensional versions of Theorem~\ref{1sample} have been proved by Dwivedi, Feldheim, Gurel-Gurevich and Ramdas \cite{DFGGR}, but up to polylogarithmic factors in the discrepancy bound.
\smallskip%

More broadly, our work falls into the class of problems that aim to improve a sample through small interventions, such as removing, adding, or replacing a small fraction of data points, or by reweighting the sample \cite{CDP, BMMS, DFGGR, boedihardjo2022private}.
\end{remark}

\section{Proof of Theorem \ref{1sample}}
We begin by reducing the problem to the case where \(X_1,\ldots,X_n\) follow the uniform distribution on \([0,1]\). Next, we apply a recursive, multiscale algorithm that partitions \([0,1]\) into dyadic subintervals. In each subinterval \(I\), the algorithm checks whether the local discrepancy exceeds \(1/m\). If so, it selectively rebalances the sample by moving points between adjacent subintervals to reduce this discrepancy. To control the discrepancy after these adjustments, the proof relies on DKW inequality \cite{D-K-W}:
\begin{equation} \label{eq: DKW}
    \Prob*{D_n > \dfrac{t}{\sqrt{n}}} 
    \leq 2e^{-2t^2},
    \quad t \ge 0,
\end{equation}
where the optimal constant \(2\) was established by Massart \cite{Massart}. Finally, the algorithm must balance the trade-off between moving only \(O(m)\) points on average and reducing the overall discrepancy to \(O(1/m)\); this is achieved via further probabilistic arguments.
\smallskip%

To begin with, let us show that it suffices to consider the case where \( X_1, \dots, X_n \) follow the uniform distribution on \( [0,1] \). This step is standard:

\begin{lemma}
If Theorem \ref{1sample} holds for the uniform distribution \( \mathrm{Unif}[0,1] \), then it holds for all distributions.
\end{lemma}
\begin{proof}
Let \( X \) be a random variable with cumulative distribution function \( F(x) \), which may not be continuous or strictly increasing. Consider the generalized inverse of \( F \):
\[
F^{-1}(u) = \inf\left\{x\,:\, F(x) \geq u \right\}.
\]
For all \( x \in \mathbb{R} \) and \( u \in [0,1] \), we have
\begin{equation}\label{gen_F}
\{ F^{-1}(u) \leq x \} \quad \text{iff} \quad \{ u \leq F(x) \}.
\end{equation}
If $X$ is $F$-distributed, then 
so is \( F^{-1}(U) \) for a uniform random variable $U$, as follows from \eqref{gen_F}:
\[
\Prob{F^{-1}(U) \leq x} 
= \Prob{U \leq F(x)} 
= F(x).
\]
Thus, we may assume without loss of generality that \( X_i = F^{-1}(U_i) \), where \( U_i \) are independent uniform random variables. By \eqref{gen_F}, for each \( x \),
\[
\#\{X_i \leq x \} = \#\{U_i \leq F(x) \}.
\]
Setting \( u = F(x) \), we find:
\[
\abs*{\frac{\#\{X_i \leq x \}}{n} - F(x)} 
= \abs*{\frac{\#\{U_i \leq u \}}{n} - u},
\]
and the lemma follows. \qed
\end{proof}
\smallskip%

Throughout, let us assume that the sample \( X_1, \dots, X_n \) is drawn from the uniform distribution \( \mathrm{Unif}[0,1] \). For convenience, let us also assume that $n$ is odd. 
If $n$ is even, one may simply discard one observation and reintroduce it later; since a single point contributes negligibly to the discrepancy, this does not affect the overall result.
\smallskip%

Our goal is to modify the sample by replacing, on average, only \(O(m)\) observations so that the new sample has discrepancy \(O(1/m)\).
\smallskip

Introduce a uniform grid:
\[
Y_k = \frac{k}{n},\quad k = 0,1,\dots,n-1.
\]
Let us use the following measure of discrepancy. For each (closed) interval \( I \subset [0,1] \), define:
\[
d_Y(I) 
= \sup_J \abs*{\frac{\#\{X_i \in J \}}{n} - \frac{\#\{Y_k \in J \}}{n}},
\]
where the supremum is taken over all intervals \( J \subset I \) sharing the left endpoint with \( I \). Note that for each \( J \), 
\[
\abs*{\frac{\#\{X_i \in J \}}{n} - \abs{J}} 
\leq \abs*{\frac{\#\{X_i \in J \}}{n} - \frac{\#\{Y_k \in J \}}{n}} + \frac{1}{n},
\]
where \( \abs{J} \) denotes the length of \( J \).
\smallskip%

We shall say that an interval $I$ is 
\emph{acceptable} if 
\[
d_Y(I) \leq \dfrac{1}{m}.
\]
We shall say that an interval $I$ is 
\emph{equalized} if 
\[
\#\{X_i \in I \} = \#\{Y_i \in I \}.
\]
Our goal is to arrange that \( [0,1] \) 
is acceptable so that
\[
D_n \leq d_{Y}([0,1]) + \frac{1}{n} \leq 
\frac{1}{m} + \frac{1}{n} \leq \frac{2}{m},
\]
while keeping the number of replaced points 
\(O(m)\) on average. To do so, let us arrange the replacements as follows. 
\smallskip

\text{Algorithm 1:} 
\begin{itemize}
\item Start with \(I=[0,1]\).
\smallskip

\item If the current interval \(I\) is \emph{not acceptable}, divide it into two equal halves, \(I_1\) and \(I_2\). Then randomly and independently select the appropriate number of sample points in one half and relocate them to random positions in the other half (chosen independently and uniformly) so that the two halves become equalized.
\smallskip

\item Repeat the previous step recursively for both \( I_1 \) and \( I_2 \).
\end{itemize}
\smallskip%

Since we start with \(I = [0,1]\), all interval endpoints generated by the recursion are dyadic 
(of the form \(2^{-j}\)). Moreover, if \(n\) is odd, no grid point \(Y_k\) belongs simultaneously to both halves \(I_1\) and \(I_2\), and almost surely no sample point \(X_i\) coincides with a dyadic endpoint.  
\smallskip

\begin{remark}[The output of Algorithm~1, formally]
Given \(X_1, \dots, X_n\) and \(m\), Algorithm~1 outputs the modified sample \(\tilde{X}_1, \dots, \tilde{X}_n\), the move count
\[
M_1 = \sum_{i=1}^n \mathbf{1}_{\{X_i \neq \tilde{X}_i\}},
\]
and the annotated dyadic recursion tree \(T_1\).  
Each node of \(T_1\) corresponds to a dyadic interval processed by the algorithm.  
For every such interval \(I\), the algorithm stores a snapshot taken when \(I\) is processed.  
This snapshot records the number and positions of the points in \(I\) at that moment (before equalization of \(I_1\) and \(I_2\)), the local discrepancy \(d_Y(I)\), and whether \(I\) is acceptable. If \(I\) is not acceptable, the snapshot also flags the points moved during equalization.  
All snapshots are immutable.
\end{remark}
\smallskip%

Algorithm~1 terminates after finitely many steps.  
Indeed, as the recursion proceeds, it reaches intervals of length less than \(1/n\), which are automatically acceptable.
\smallskip%

Let us now verify that the final outcome makes \([0,1]\) acceptable.
\begin{lemma}
At the end of Algorithm 1, the interval \([0,1]\) is acceptable.
\end{lemma}
\begin{proof}
When the algorithm stops, it generates a partition of $[0,1]$ into equalized, acceptable intervals. For any \(x\in[0,1]\), let \(I\) be the unique interval (generated during the algorithm) that contains \(x\), and let \(y\) denote the left endpoint of \(I\). Recall that we assume \(n\) to be 
odd. Hence, no grid point \(Y_i\) lies in two dyadic 
intervals of the same length. Moreover, almost surely, none of the sample points \(X_i\), nor any points moved during the algorithm, lie exactly at dyadic boundary points \(2^{-j}\).
\smallskip

We calculate:
\[
\left| 
\frac{\#\{X_i \leq x \}}{n} - 
\frac{\#\{Y_i \leq x \}}{n} 
\right| \leq 
\left| 
\frac{\#\{X_i \leq y \}}{n} - 
\frac{\#\{Y_i \leq y \}}{n} \right| + 
\left| 
\frac{\#\{X_i \in [y,x] \}}{n} - 
\frac{\#\{Y_i \in [y,x] \}}{n} \right|
\]
Since \([0,y]\) is a disjoint union of back-to-back equalized intervals, it itself is equalized, so the first term on the right hand side is zero. Since \(I\) is acceptable, the second term is at most \(1/m\). Consequently,
\[
\left| 
\frac{\#\{X_i \leq x \}}{n} - 
\frac{\#\{Y_i \leq x \}}{n} \right| \leq 0 + \frac{1}{m} \leq 
\frac{1}{m},
\]
and the lemma follows. \qed
\end{proof}
\smallskip%

Let \(M_1\) denote the total number of points moved by Algorithm 1. We shall show that \(M_1\) is \(O(m)\) on average. To bound the expected total number of replaced points, we compare Algorithm 1 to the following alternative algorithm:  
\smallskip%

\text{Algorithm 2:} 
\begin{itemize}
\item Start with \(I=[0,1]\).
\smallskip

\item Divide $I$ into two equal halves, \(I_1\) and \(I_2\). Then randomly and independently select the appropriate number of sample points in one half and relocate them to random positions in the other half (chosen independently and uniformly) so that the two halves become equalized.
\smallskip

\item Repeat the previous step recursively for both \( I_1 \) and \( I_2 \).
\end{itemize}
\smallskip%

Algorithm~2 is identical to Algorithm~1 except that it divides every interval encountered, regardless of acceptability. Algorithm~2 never terminates.

\begin{remark}[The output of Algorithm~2, formally]
Given \(X_1, \dots, X_n\) and \(m\), Algorithm~2 produces the annotated dyadic recursion tree \(T_2\). Its nodes correspond to the dyadic intervals processed by the algorithm. For each interval \(I\), the algorithm records a snapshot at the time \(I\) is processed.  
This snapshot contains the number and positions of points in \(I\) (before equalization of \(I_1\) and \(I_2\)), the local discrepancy \(d_Y(I)\), and whether \(I\) is acceptable. If \(I\) is not acceptable, the snapshot also flags the points moved during equalization. All snapshots are immutable.
\smallskip%

Both recursion trees \(T_1\) and \(T_2\) are indexed by the dyadic intervals of \([0,1]\). The tree \(T_1\) is a subtree of \(T_2\), and the two algorithms produce identical snapshots on their common nodes.
\smallskip%

Statements of Lemma~\ref{tail_lemma} and Lemma~\ref{m(I)} referring to an interval \(I\) (and to quantities such as \(d_Y(I)\)) are to be understood with respect to the snapshot associated with the node \(I \in T_2\).  
Both lemmas rely on the following key observation: when Algorithm~2 processes an interval \(I\), the points \(X_i\) falling in \(I\), even if have been moved during earlier steps, remain independent and uniformly distributed within \(I\). 
\end{remark}
\smallskip%

Although Algorithm~2 never stops dividing, it processes only finitely many unacceptable intervals. Indeed, once intervals become shorter than \(1/n\), they are acceptable.
\smallskip%

Let \(M_2\) denote the total number of points moved by Algorithm~2 while processing unacceptable intervals.  
\begin{lemma}\label{M1M2}
Let $M_1$ denote the total number of points moved by Algorithm 1. Let \(M_2\) denote the total number of points moved by Algorithm 2 in those intervals that are unacceptable.
If both algorithms are applied to the same input sample using the same random choices, then
\[
M_1 \le M_2.
\]
\end{lemma}
\begin{proof}
Up until an acceptable interval, both algorithms perform the same replacements. However, when an acceptable interval is encountered, Algorithm 1 stops processing that interval while Algorithm 2 continues to subdivide and perform additional relocations. \qed
\end{proof} 
\smallskip%

At every stage in Algorithm 2, the interval \([0,1]\) is partitioned into equalized subintervals. Within each such interval, the points $X_i$ remain independent and uniformly distributed. With this understood, 
we claim:
\begin{lemma}\label{tail_lemma}
Let \( I \) be an interval of length \( |I| = 1/k \) encountered in Algorithm 2, and define:
\[
\delta(I) = d_Y(I) \cdot 
\mathbf{1}_{
\{ I \text{\normalfont{ is unacceptable}} 
\}}.
\]
Then for each $\lambda \geq 0$, 
\begin{equation}\label{tail}
\Prob*{\delta(I) > \frac{\lambda}{\sqrt{n\,k}}} 
\leq C_1 e^{-\lambda^2/2}
\end{equation}
for some absolute constant $C_1 > 0$. In particular, 
\begin{equation}\label{E}
\E \delta(I) 
\leq \frac{C_2}{\sqrt{n\,k}}
\end{equation}
for some absolute constant $C_2 > 0$.
\end{lemma}

\begin{proof}
Let \(n_0=\#\{X_i\in I\}\) at the time \(I\) is processed. Since \(I\) is equalized, \(n_0\) equals the number of grid points \(\{0,1/n,\dots,(n-1)/n\}\) in \(I\). Hence,
\begin{equation}\label{eq:n0-two-sided}
\Bigl|\,n_0 - \frac{n}{k}\Bigr|\le 1.
\end{equation}
If $k \geq n$, the interval $I$ has length $1/k \leq 1/n$ and so $n_0 \le 1$. With either zero or one sample point, the discrepancy trivially satisfies $d_Y(I) \le 1/n \le 1/m$, so $I$ is automatically acceptable.
\smallskip%

Consider the case where \(k<n\) now. Without loss of generality, assume that \(I = [0,1/k] \). DKW inequality \eqref{eq: DKW} for $n_0$ iid points sampled from the uniform distribution on $[0,1/k]$ implies that with probability at least $1-2e^{-2t^2}$, the following ``good event'' holds:
\[
\left|
\#\{X_i \leq x \} - n_0\,x\,k 
\right| \leq t \sqrt{n_0}
\quad \text{for any } x \in [0,1/k].
\]
On the other hand, by triangle inequality,
\begin{multline*}
\left|
\#\{X_i \leq x \} - n\,x 
\right| \leq 
\left|
\#\{X_i \leq x \} - n_0\,x\,k 
\right| + 
\left|
n_0\,x\,k - n\, x
\right| = \\
 = \left|
\#\{X_i \leq x \} - n_0\,x\,k 
\right| + 
x\,k \left|
n_0 - n/k
\right| \stackrel{\eqref{eq:n0-two-sided}}{\leq} 
\left|
\#\{X_i \leq x \} - n_0\,x\,k 
\right| + 
x\,k \leq \\
\leq 
\left|
\#\{X_i \leq x \} - n_0\,x\,k 
\right| + 1.
\end{multline*}
Then, on the good event, we get:
\[
\left|
\#\{X_i \leq x \} - n\,x 
\right| \leq t \sqrt{n_0} + 1.
\]
On the other hand, the grid points $Y_i$ trivially satisfy:
\[
\abs*{\#\{Y_i \leq x\} - n\,x}
\le 1.
\]
Since these two bounds hold for all $x \in I$, combining them gives:
\begin{align*}
n\, d_Y(I)
&= \sup_{x \in I} \abs*{\#\{X_i \leq x\} - \#\{Y_i \leq x\}}
\le t \sqrt{n_0} + 2
\stackrel{}{\le}
t \sqrt{2n/k} + 2 \\
&\le t \sqrt{2n/k} + 2 \sqrt{n/k}
    \quad \text{(by the assumption $k<n$)}\\
&= (t\sqrt{2}+2)\sqrt{n/k}
= \lambda \sqrt{n/k},
\end{align*}
where we can assume that $\lambda \ge 2$ to ensure that $t \ge 0$ (otherwise we can make the  bound \eqref{tail} trivial by adjusting $C_1$.)

Summarizing, we showed that if the good event occurs, then $d_Y(I) \le \lambda/\sqrt{nk}$, and the good event occurs with probability at least  
\[
1 - 2 e^{-2\,t^2} = 
1 - 2 e^{-(\lambda-2)^2} 
\geq 1 - C_1 e^{-\lambda^2/2}
\]
for some $C_1$. This completes the proof. \qed
\end{proof}
\smallskip%

Let \(k_0\) be the (not necessarily integer) number such that 
\[
\sqrt{n k_0} = m.
\]
It then suffices to show that
\[
\E M_2 = O\bigl(\sqrt{n\,k_0}\bigr),
\]
since, by Lemma \ref{M1M2}, the bound on \(\E M_1\) then follows. 
\smallskip%

For each interval \(I\) encountered in Algorithm 2, let
\[
m(I) = \#\{\text{points moved between 
$I_1$ and $I_2$ by Algorithm 2}\}.
\]
\begin{lemma}\label{m(I)}
For each interval \(I\) encountered in Algorithm 2,
\[
m(I) \le 2\,n\,d_Y(I) + 1.
\]
\end{lemma}
\begin{proof}
Since \(I\) is equalized, it follows 
that
\[
m(I) = \left| 
\#\{X_i \in I_1 \} - \#\{X_i \in I_2 \}
\right| = 
\left| 
2 \left(
\#\{X_i \in I_1 \} - \#\{Y_i \in I_1 \}
\right) + 
\#\{Y_i \in I_1 \} - \#\{Y_i \in I_2 \}
\right|.
\]
Consequently,
\[
m(I) \leq 2\,n\,d_{Y}(I_1) + 
\left|
\#\{Y_i \in I_1 \} - \#\{Y_i \in I_2 \}
\right| \leq 
2\,n\,d_{Y}(I) + 1,
\]
since $I_1$ and $I_2$ have the same length. \qed
\end{proof}
\smallskip%

Each interval \(I\) in Algorithm~2 has length \(|I| = 1/k\) for some dyadic integer \(k \in \{1,2,4,\ldots\}\). Hence, the total number of replacements is given by
\[
M_2 \leq \sum_{k = 1, 2, 4, \ldots} \; \sum_{\substack{|I| = 1/k}} m(I) \cdot \mathbf{1}_{\{ I \text{ is unacceptable} \}}.
\]
Here the second summation is over 
dyadic intervals of the form 
\([j/k, (j+1)/k], j = 0,1,\ldots, k-1\).
\smallskip%

For each unacceptable $I$, we have $nd_Y(I) > n/m \ge 1$, so Lemma \ref{m(I)} gives $m(I) \le 3\,n\,d_{Y}(I)$. Then, for each \(I\) with \(|I| = 1/k\), we have
\[
\mathbb{E}[
m(I) 
\cdot \mathbf{1}_{
\{ I \text{ is unacceptable}\}
}
] 
\leq
\mathbb{E}[
3\,n\, 
d_{Y}(I) 
\cdot \mathbf{1}_{\{ I 
\text{\normalfont{ is unacceptable}} \}}
] = 
3\,n\,\mathbb{E}[\delta(I)]
\stackrel{\eqref{E}}{\leq} 3\,C_2 \sqrt{n/k}.
\]
It is now convenient to treat the cases \(k \le k_0\) and \(k > k_0\) separately. First consider the case when $k \leq k_0$. Since the algorithm encounters at most $k$ 
intervals of length $1/k$, it follows that
\[
\mathbb{E}\left[ 
\sum_{k \leq k_0} \sum_{|I| = 1/k} m(I) 
\cdot \mathbf{1}_{\{ I \text{\normalfont{ is unacceptable}}\}}
\right] \leq 
\sum_{k \le k_0} 3\,C_2 \sqrt{n k} = 
3\,C_2 \sqrt{n} \sum_{k \le k_0} \sqrt{k} = 
O(\sqrt{n k_0}),
\]
where the latter equality follows since 
the sum is over dyadic values of \(k\). Next, consider the case when $k > k_0$. Since
\[
m(I) 
\cdot \mathbf{1}_{\{ I \text{\normalfont{ is unacceptable}}\}} \leq 
3\,n\,d_Y(I) \cdot \mathbf{1}_{\{ I \text{\normalfont{ is unacceptable}}\}} = 
3\,n\,\delta(I) \cdot \mathbf{1}_{\{ I \text{\normalfont{ is unacceptable}}\}},
\]
it follows from Cauchy–Schwarz inequality 
that for each 
\(I\) with \(|I| = 1/k\),
\[
\mathbb{E}[
m(I) 
\cdot \mathbf{1}_{\{ I \text{\normalfont{ is unacceptable}}\}}
] \leq 
3\,n
\mathbb{E}[
\delta(I)^2 
]^{1/2} \cdot 
\mathbb{E}[
\mathbf{1}_{\{ I \text{\normalfont{ is unacceptable}}\}}
]^{1/2}.
\]
From the tail bound \eqref{tail}, we find for each \(I\) with \(|I|=1/k\),
\begin{equation}\label{E^2}
\mathbb{E}[\delta(I)^2] \leq \frac{C_3}{n k}
\end{equation}
for some constant $C_3 > 0$. On the other hand, for $k > k_0$, 
\[
\mathbb{E}[
\mathbf{1}_{\{ I \text{\normalfont{ is unacceptable}}\}} 
]^{1/2} 
= \Prob*{ \delta(I) > \frac{1}{\sqrt{n k_0}} }^{1/2} = 
\Prob*{ \delta(I) > \frac{\sqrt{k/k_0}}{\sqrt{n k}} }^{1/2} \stackrel{\eqref{tail}}{\leq} 
\sqrt{C_1} 
e^{-(k/k_0)/4},
\]
where we used \eqref{tail} 
with \(\lambda = \sqrt{k/k_0}\). Consequently, for each $k > k_0$,
\[
\mathbb{E}[
m(I) 
\cdot \mathbf{1}_{\{ I \text{\normalfont{ is unacceptable}}\}}
] \leq 3 \sqrt{C_1\, C_3} \sqrt{n/k} e^{-(k/k_0)/4}.
\]
Since the algorithm encounters at most $k$ intervals of length $1/k$, summing over all of them and then over numbers $k>k_0$ that are powers of $2$, we get
\[
\mathbb{E}\left[ 
\sum_{k > k_0} \sum_{|I| = 1/k} m(I) 
\cdot \mathbf{1}_{\{ I \text{\normalfont{ is unacceptable}}\}}
\right] \leq 
3 \sqrt{C_1\, C_3} \sqrt{n} \sum_{k > k_0} \sqrt{k} 
e^{-(k/k_0)/4} = O(\sqrt{n k_0}).
\]
Hence, adding the contributions from both ranges of \(k\), we find:
\[
\E M_2 = O(\sqrt{n k_0}) + O(\sqrt{n k_0}) = 
O(\sqrt{n k_0}),
\]
and the proof of Theorem \ref{1sample} is complete.

\section{Complexity of Algorithm 1}\label{complexity}

The time complexity of Algorithm 1 is $O\left(n \ln(n) \ln (m)\right)$.
To check this, recall that each \emph{equalized} interval 
\(I\) with $|I| \leq 1/(2m)$ is necessarily acceptable. Indeed, if \(m \leq n/2\), we find:
\[
d_{Y}(I) 
\le \left| 
\frac{\#\{X_i \in I \}}{n}
\right|
=
\left| 
\frac{\#\{Y_i \in I \}}{n}
\right| \le |I| + \frac{1}{n} \leq 
\frac{1}{2 m} + \frac{1}{n} \leq 
\frac{1}{m}.
\] 
On the other hand, if \(m > n/2\), then \(|I| < 1/n\). 
Hence, \(d_{Y}(I) \leq 1/n \leq 1/m\). 
\smallskip%

Consequently, the recursion stops after at most \(O(\ln{m})\) steps. On each step, computing the empirical distribution function requires \(O(n \ln n)\) operations, and the replacement of the sample points is done in linear time. Combining these yields the overall time complexity.

\section{Proof of Proposition \ref{lower_bound}}
One can never achieve discrepancy less 
than \(1/(2n)\), so we may assume 
\(c/m \geq 1/(2n)\). This implies \(m \leq 2 c n\) 
and \(k = m^2/n\leq 4 c^2 n\), and $k$ is integer by assumption. Partition the interval \([0,1]\) into \(k\) subintervals \(I_1, I_2, \ldots, I_k\) of equal length \(1/k\). For each \(I_j\), 
set:
\[
\nu_j = \#\{X_i \in I_j\}.
\]
We shall say that $I_j$ is \emph{regular} if 
$$
\nu_j \le n/k + 0.5 \sqrt{n/k}.
$$
In order to achieve discrepancy $D_n < 1/(4m)$, all intervals $I_j$ must be regular. Indeed, by triangle inequality we have
$$
\abs*{ \frac{\#\{X_i \in I_j\}}{n} - \abs{I_j} } 
\le 2D_n
$$
so if $I_j$ is irregular then we would have
$$
2D_n \ge \frac{\nu_j}{n} - \frac{1}{k}
\ge \frac{0.5}{\sqrt{nk}}
= \frac{1}{2m}.
$$

Note that $\nu_j$ is a binomial random variable with mean $n/k \ge 1/(4 c^{2}) \ge 100$ for \(c\) small enough. Call an interval $I_j$ \emph{dense} if 
$$
\nu_j > n/k + \sqrt{n/k}.
$$
Standard results (e.g., via the Berry--Esseen theorem or direct lower bounds for the binomial tail) imply that
\[
\Prob{\text{$I_j$ is dense}} > \frac{1}{100}.
\]
Thus, the expected number of dense intervals is at least \(k/100\). To achieve the discrepancy $D_n < 1/(4m)$, we must make each interval regular, and this requires moving at least $0.5\sqrt{n/k}$ points out of each dense interval. Thus, achieving $D_n < 1/(4m)$ requires moving, on average, at least 
\[
\frac{k}{100} \cdot 0.5\sqrt{n/k} 
= \frac{\sqrt{nk}}{200} 
= \frac{m}{200}
\]
points. This completes the proof.

\bibliographystyle{plain}
\bibliography{ref}

\end{document}